\documentclass[12pt]{amsart}  
\usepackage{fullpage, amsmath,  amssymb, latexsym, stmaryrd, color,mathtools}
\usepackage{color}
\usepackage{tikz}
\usepackage{mathrsfs}

\definecolor{hot}{RGB}{65,105,225}
\usepackage[pagebackref=true,colorlinks=true, linkcolor=hot ,  citecolor=hot, urlcolor=hot]{hyperref}

\theoremstyle{plain}

\newtheorem{theorem}{Theorem}[section]
\newtheorem{lemma}[theorem]{Lemma}
\newtheorem{proposition}[theorem]{Proposition}
\newtheorem{prop-def}[theorem]{Proposition-Definition}
\newtheorem{corollary}[theorem]{Corollary}

\theoremstyle{definition}

\newtheorem{defn}[theorem]{Definition}

\newtheorem{remark}[theorem]{Remark}
\newtheorem{rmk}[theorem]{Remark}

\theoremstyle{remark}

\usepackage{graphicx}
\usepackage[font=footnotesize,labelfont=bf]{caption}
\usepackage{array, tabularx, longtable}
\usepackage{color}
\numberwithin{equation}{section}

\def\m{\mathfrak m}

\def\n{\mathbf n}
\def\bC{\mathbb{C}}
\def\bP{\mathbb{P}}
\def\bA{\mathbb{A}}
\def\cI{\mathcal{J}}

\def\ord{\mathrm{ord}}

\def\Spec{\mathrm{Spec}}

\def\n{\mathbf{n}}
\def\L{\mathbb{L}}

\def\l{\mathbf{l}}

\def\A{\mathbb{A}}
\def\O{\mathcal{O}}

\DeclareMathOperator{\N}{\mathbb N}
\DeclareMathOperator{\Z}{\mathbb Z}
\DeclareMathOperator{\C}{\mathbb C}

\DeclareMathOperator{\\k}{\mathbb k}
\DeclareMathOperator{\Cont}{Cont}
\DeclareMathOperator{\rank}{rank}

\DeclareMathOperator{\codim}{codim}

\DeclareMathOperator{\X}{\mathcal{X}}

\DeclareMathOperator{\LL}{\mathcal{L}}
\DeclareMathOperator{\Aa}{\mathscr{A}}
\DeclareMathOperator{\Bb}{\mathscr{B}}
\DeclareMathOperator{\XX}{\mathfrak{X}}
\DeclareMathOperator{\G}{\mathcal{G}}
\DeclareMathOperator{\B}{\mathcal{B}}

\def\cM{\mathcal{M}}
\def\ra{\rightarrow}
\def\be{\begin{equation}}
\def\ee{\end{equation}}

\makeatletter
\def\subsection{\@startsection{subsection}{3}%
  \z@{.5\linespacing\@plus.7\linespacing}{.1\linespacing}%
  {\normalfont\bfseries}}
\makeatother  

\title[On contact loci of hyperplane arrangements]{\bf On contact loci of hyperplane arrangements}  

\author{Nero Budur}
\address{Department of Mathematics, KU Leuven, Celestijnenlaan 200B, 3001 Leuven, Belgium}
\email{nero.budur@kuleuven.be}
\author{Tran Quang Tue}
\address{Department of Mathematics, KU Leuven, Celestijnenlaan 200B, 3001 Leuven, Belgium}
\curraddr{Faculty of Fundamental Sciences, Phenikaa University, Yen Nghia, Ha Dong, Hanoi 12116, Vietnam.}
\email{tranquangtue.math@gmail.com}


\begin{document}           

\begin{abstract}
{We give an explicit expression for the contact loci of hyperplane arrangements and show that their cohomology rings are combinatorial invariants. We also give an expression for the restricted contact loci in terms of Milnor fibers of associated hyperplane arrangements. We prove the degeneracy of a spectral sequence related to the restricted contact loci of a hyperplane arrangement and which conjecturally computes algebraically the Floer cohomology of iterates of the Milnor monodromy. We give formulas for the Betti numbers of contact loci and restricted contact loci in generic cases.}
\end{abstract}

\maketitle                 

\tableofcontents

\section{Introduction}\label{secintro}
Let $X$ be a smooth variety of dimension $n\ge 1$ over the field of complex numbers, and $f$ be a non-constant morphism from $X$ to the complex affine line $\A^1$. We denote by $\LL_m(X)$ the space of $m$-jets on $X$. The {\it $m$-contact locus} $\X_m(f)$ of $f$ is defined to be the set of $m$-jets $\gamma\in \LL_m(X)$ with order of vanishing along $f$ precisely $m$. 


The contact loci appear in the definition of the naive motivic zeta function
$$Z^{\text{naive}}(f,T)=\sum_{m\geq 0}[\X_m(f)]\L^{-nm}T^m,$$
defined over $\cM=K_0(Var_{\C})[\L^{-1}]$, the localization at $\L=[\A^1]$ of the Grothendieck ring of complex varieties, see \cite{DL}. 

A general stratification result by locally closed subsets for contact loci can be found in \cite[Theorem A]{EiLaMu04}. This is useful for determining additive invariants such as the class in the Grothendieck ring or other Euler-characteristic type of invariants. However it is in general very difficult to describe the topology of contact loci, for example, to determine their cohomology rings. In this article we address this issue for hyperplane multi-arrangements. We provide for example a combinatorial answer for the cohomology rings of the contact loci in this case.

We let therefore $X=\A^n$ and $f=h_1^{s_1}\ldots h_d^{s_d}$, where $n, d, s_i\ge 1$ are  integers, and $h_i$ are polynomials of degree $1$, pairwise distinct up to multiplying by a constant. We let $H_i$ denote the zero set of $h_i$ in $\A^n$. We denote by $\Aa$ the associated  hyperplane multi-arrangement, that is,  the set of hyperplanes $\{H_1,\ldots, H_d\}$ together with the multiplicity function $s(H_i)= s_i$ on it. We also use the notation $\X_{m}(\mathscr{A})$ for $\X_{m}(f)$. 

An {\it edge} of $\Aa$ is a non-empty intersection of hyperplanes in $\Aa$. By the {\it intersection lattice} $L(\Aa)$ we mean the set containing all edges and $X$, together with the inclusion relations. If $\Aa$ is  {\it central}, that is if each $h_i$ is homogeneous, then $L(\Aa)$ is indeed a lattice; in general $L(\Aa)$ is only a semi-lattice, see \cite[2.2]{Dimca17}. The multiplicity function $s$ extends to the intersection lattice by setting $s(Z)=\sum_{H_i\supset Z}s_i$ for an edge $Z$ and $s(X)=0$. The minimal edges with respect to inclusion have all the same codimension; this is called the  {\it rank} of $\Aa$. By the {\it combinatorial type} of a  hyperplane multi-arrangement, we mean the data consisting of the intersection lattice  and the multiplicity function. Two hyperplane multi-arrangements are {\it combinatorially equivalent} if they have the same combinatorial type.

We state now the main result for central multi-arrangements $\Aa$. For non-central multi-arrangements, see Theorem \ref{thmMain2}.


\begin{theorem}\label{thmMain}
Let $\Aa$ be a central hyperplane multi-arrangement in $\A^n$ and $m\in\N$. Then:

(i) The contact locus $\X_m(\Aa)$ admits a disjoint decomposition 
$$\X_{m}(\Aa)=\bigsqcup_{j\in S(m)} \mathscr{C}_{j}(\Aa)$$
where $S(m)=\{j\in\N^d\ |\ j_1s_1+\ldots+j_ds_d=m\}$.

(ii) If $\mathscr{C}_{j}(\Aa)\neq\emptyset$, then $\mathscr{C}_{j}(\Aa)$ is irreducible open and closed in the Zariski topology, and equals the complement of a hyperplane arrangement $\Aa_{j}$ of same rank as $\Aa$ in a linear space $X_j$.

(iii) Define $$T(m)=\{j \in S(m)\mid \mathscr{C}_{j}(\Aa)\neq\emptyset\}.$$ 

For $m=0$,  $X_j=X$, $\Aa_j=\Aa$ for $j=(0,\ldots,0)$, so that $\X_0(\Aa)=\mathscr{C}_0(\Aa)=X\setminus A$.

For $m>0$, there is a 1-1 correspondence between $T(m)$ and the set of chains of elements of the intersection lattice of $\Aa$ 
$$
Z_0\subset Z_1\subset \ldots \subset Z_m=\bA^n
$$
such that $\sum_{k=0}^m s(Z_k)=m$. The correspondence is
$$
j\mapsto Z_\bullet \text{ with } Z_k = \bigcap_{\{i\mid j_i>k \}} H_i\text{ for }0\le k\le m,
$$
and, conversely,
$$
Z_\bullet \mapsto j\text{ with } j_i=\#\{k\mid H_i\supset Z_k\} \text{ for }1\le i\le d.
$$
Under this correspondence,
$$
X_j = Z_0\times\ldots\times Z_m
$$
and the hyperplane arrangement $\Aa_j$ in $X_j$
is the product arrangement
$$
\Aa_j = \Aa_{j,0}\times\ldots\times \Aa_{j,m} $$
where $\Aa_{j,k}$ is the restriction to  $Z_k$ 
of  
$$
\mathop{\bigcup_{H_i\supset Z_{k-1}}}_{\text{and }H_i\not\supset Z_k} H_i
$$
for $k\ge 0$, with  $Z_{-1}=\emptyset$.
In particular,
$$
\mathscr{C}_j(\Aa) = X_j\setminus\Aa_j= \times_{k=0}^m (Z_k\setminus \Aa_{j,k}).
$$

(iv) The combinatorial type of $\Aa$ determines $T(m)$ and the combinatorial type of each $\Aa_j$. The dimension function on the intersection lattice of $\Aa$ determines further the dimension function on the intersection lattice of each $\Aa_j$.
\end{theorem}

Theorem \ref{thmMain} is a refinement for this particular case of decomposition formulas of \cite[Theorem A]{EiLaMu04} and \cite{Mus06}, where the non-empty terms are not explicitly determined, nor an explicit determination in terms of associated hyperplane arrangements is given, see Proposition \ref{propCompELM}.

Denote by
$$
b_k(\X_m(\Aa)):=\textup{rank}\, H_k(\X_m(\Aa),\Z)
$$
the $k$-th Betti number of $\X_m(\Aa)$, and by
$$
B(\X_m(\Aa),t) := \sum_k b_k(\X_m(\Aa))t^k
$$
the Betti polynomial. Then Theorem \ref{thmMain} (ii), and its generalization to the non-central case, Theorem \ref{thmMain2} (ii), together with the Orlik-Solomon theorem \cite[Corollary 3.6]{Dimca17} imply:

\begin{corollary} Let $\Aa$ be a hyperplane multi-arrangement and $m\in\N$.
The degree of the Betti polynomial of the contact locus $\X_m(\mathscr{A})$ coincides with the rank of $\Aa$. 
\end{corollary} 
 
 Furthermore, based on the decomposition from the main theorem we can compute explicitly the cohomology of the contact loci of $\Aa$ in terms of the combinatorial type of $\Aa$, see Proposition \ref{cohexp}. A direct consequence is:

\begin{theorem}\label{thmIsom}
Let $\mathscr{A}$ and $\mathscr{B}$ be two hyperplane multi-arrangements in ambient affine spaces of possibly different dimensions, not necessarily central. Let $R$ be a unitary commutative ring. If $\mathscr{A}$ and $\mathscr{B}$ are combinatorially equivalent, then the cohomology algebras of their $m$-contact loci $H^*(\X_m(\mathscr{A}),R)$ and $H^*(\X_m(\mathscr{B}),R)$ are isomorphic as graded $R$-algebras for every $m\in\N$.
\end{theorem}

We compute the Betti numbers of contact loci for generic and, respectively, generic central hyperplane arrangements. Throughout the article we let $\binom{a}{b}$ be the usual binomial coefficient if $0\le b\le a$ are integers, and otherwise we define it to be zero. A sum $\sum_{i=b}^az_i$ will be considered zero if the integers $a$ and $b$ do not satisfy that $b\le a$. 

\begin{theorem}\label{betgenarr}
Let $\Aa$ be a generic arrangement in $\A^n$ of $d\ge 1$ hyperplanes. Let $m\in \N^*$. Then $$b_k(\X_m(\Aa))=\begin{pmatrix}d\\k\end{pmatrix}\sum_{i=0}^{n-k} \begin{pmatrix}d-k\\i\end{pmatrix}\begin{pmatrix}m+k-1\\m-i\end{pmatrix}.$$
In particular, the number of irreducible components of $\X_{m}(\Aa)$ is $$b_0(\X_m(\Aa))=\sum_{i=0}^{n} \begin{pmatrix}d\\i\end{pmatrix}\begin{pmatrix}m-1\\m-i\end{pmatrix}.$$
\end{theorem}

\begin{theorem}\label{thmGCb} Let $\Aa$ be a generic central arrangement in $\A^n$ of $d>n$ hyperplanes. Let $m\in \N^*$. Then for $0\leq k\leq n$,
\begin{align*}
b_k(\X_m(\Aa)) & =  \sum_{j=0}^{\lceil\frac{m}{d}\rceil} \left[\begin{pmatrix}d\\k\end{pmatrix}\sum_{i=0}^{n-1-k}\begin{pmatrix}d-k\\i\end{pmatrix}\begin{pmatrix}m-jd+k-1\\i+k-1\end{pmatrix}+\right.\\
& \left. +\sum_{i=1}^{n-1}\begin{pmatrix}d\\i\end{pmatrix}\begin{pmatrix}m-jd-1\\i-1\end{pmatrix}\begin{pmatrix}d-i-1\\d-n\end{pmatrix}\begin{pmatrix}i\\n-k\end{pmatrix}\right]+\delta_{k,0}\epsilon_{d,m}- \binom{d-1}{n}\delta_{k,n}\epsilon_{d,m}
\end{align*}
where $\delta_{a,b}$ denotes the Kronecker delta for every $a,b\geq 0$, and $\epsilon_{d,m}$ is 1 if $d$ divides $m$ and 0 otherwise. 

In particular, the number of irreducible components of $\X_{m}(\Aa)$  is 
$$b_0(\X_m(\Aa))=\sum_{j=0}^{\lceil \frac{m}{d}\rceil}\sum_{i=1}^{n-1}\begin{pmatrix}d\\i\end{pmatrix} \begin{pmatrix}m-j d-1\\i-1\end{pmatrix}+\epsilon_{d,m}.$$
\end{theorem}

In addition to the $m$-contact locus $\X_m(\Aa)$, we also consider the {\it restricted $m$-contact locus} $\XX_m(f)$ of a defining polynomial, consisting of $m$-jets with order of contact $m$ with $f$ and angular component precisely 1, see Definition \ref{defCL}. The geometry of the restricted contact loci is richer than that of the contact loci, and they give rise to the non-naive version of the motivic zeta function
$$Z(f,T)=\sum_{m\geq 0}[\XX_m(f)]\L^{-nm}T^m$$
defined  over  $\cM^{\hat{\mu}}=\varinjlim_k K_0^{\mu_k}(Var_\bC)[\L^{-1}]$, where $\mu_k$ is the group of $k$-th roots of unity and  $K_0(Var_\bC)$ is the $\mu_k$-equivariant Grothendieck ring of complex varieties, see \cite{DL}. The cohomology groups with compact supports of the restricted $m$-contact loci are conjecturally related to the Floer cohomology groups of the $m$-th iteration of the monodromy on the Milnor fiber, see \cite{Floer}.

If the multi-arrangement is central, then $\XX_m(f)$ does not depend on the choice of a defining polynomial for the multi-arrangement $\Aa$, in which case we set $\XX_m(\Aa)=\XX_m(f)$. For simplicity we state here our result for the central case only; for the general case see Theorem \ref{prorescon}.

\begin{theorem}\label{prorescon2} Let $\Aa$ be as in Theorem \ref{thmMain}, and $m\in\N$. There is a natural multi-arrangement structure on each hyperplane arrangement $\Aa_j$, $j\in T(m)$, from Theorem \ref{thmMain}, such that there is a disjoint decomposition into Zariski open and closed irreducible subsets of the restricted contact locus  
$$\XX_{m}(\Aa)=\bigsqcup_{j\in T(m)}\mathscr{F}_j(\Aa)$$
where $\mathscr{F}_j(\Aa)$ is the Milnor fiber of $\Aa_j$. 

In particular, the degree of the Betti polynomial of $\XX_{m}(\Aa)$ is the rank of $\Aa$ minus one.
\end{theorem}

We show that for hyperplane arrangements, the spectral sequence of \cite{Floer} converging to the cohomology with compact supports of the restricted contact locus degenerates at the first page, see Proposition \ref{propFlDeg}. According to a conjecture from \cite{Floer}, the same spectral sequence should compute the Floer cohomology of the iterates of the Milnor monodromy.

We compute the Betti numbers of restricted contact loci of generic central hyperplane arrangements:

\begin{theorem}\label{thmRCGA} Let $\Aa$ be a generic central arrangement in $\A^n$ of $d>n$ hyperplanes. Let $m\in \N^*$. Then for $0\leq k\leq n-1$,
\begin{align*}
b_k(\XX_m(\Aa)) = &\sum_{j=0}^{\lceil \frac{m}{d}\rceil}\sum_{l=1}^{n-1}\begin{pmatrix}d\\l\end{pmatrix}\begin{pmatrix}m-jd-1\\l-1\end{pmatrix}\sum_{i=0}^{n-1-l}\begin{pmatrix}d-1-l\\i\end{pmatrix}\begin{pmatrix}l\\k-i\end{pmatrix}+\\
&+  \binom{d-1}{k} \epsilon_{d,m} +(d-1)\binom{d-2}{n-1}\delta_{k,n-1}\epsilon_{d,m}
\end{align*}
where $\delta_{k,n-1}$ and $\epsilon_{d,m}$ are as above. 

In particular, the number of irreducible components of $\XX_m(\Aa)$ is
$$
b_0(\XX_m(\Aa)) =\sum_{j=0}^{\lceil \frac{m}{d}\rceil}\sum_{l=1}^{n-1}\begin{pmatrix}d\\l\end{pmatrix}\begin{pmatrix}m-jd-1\\l-1\end{pmatrix}  + \epsilon_{d,m}.
$$

\end{theorem}

By Aluffi \cite[Theorem 1.1]{Al}, the class in $K_0(Var_\bC)$ of the complement of a central hyperplane arrangement $\Aa$ is 
$
\chi_{\Aa}(\L),
$ 
where $\chi_{\Aa}$ is the characteristic polynomial of the arrangement. 
Hence Theorems \ref{thmMain} and \ref{prorescon2} imply the following for the (naive) motivic zeta functions: 

\begin{theorem}\label{thmZetas}
Let $\Aa$ be a central hyperplane multi-arrangement in $\bA^n$. Then 
$$Z^{\text{naive}}(\Aa,T)=\sum_{m\geq0}
\sum_{j\in T(m)}\chi_{\Aa_j}(\L)\L^{-nm}T^m \in \cM[[T]],$$
and
$$Z(\Aa,T)=\sum_{m\geq0}
\sum_{j\in T(m)}[\mathscr{F}_j(\Aa)]\L^{-nm}T^m  \in \cM^{\hat{\mu}}[[T]],$$
with $T(m)$ and $\Aa_j$ as in Theorem \ref{thmMain}. 

\end{theorem}

\begin{rmk} We state without proof the following remark, for which we thank R. van der Veer. Namely, Theorem \ref{thmZetas} recovers in a non-tropical way \cite[Theorems 1.2 and 1.7]{KU}. Moreover, since the formula given for $Z^{\text{naive}}(\Aa,T)$ depends only on the combinatorial type of $\Aa$ and the ambient dimension $n$, one can use this formula with $q$, $\Z[q^{\pm 1}]$ replacing $[\L]$, $\cM$, respectively,  to obtain a well-defined motivic zeta function of a matroid. This coincides with the one defined by \cite[Definition 1.1]{JKU}, from which one also obtains the topological zeta function of a matroid of \cite{vV}. 
\end{rmk}

In Section \ref{secPre} we give preliminaries on jets, arcs, and contact loci. In Section \ref{seccont} we prove Theorem \ref{thmMain} and Theorem \ref{thmIsom}. In Section \ref{secRCL} we prove Theorem \ref{prorescon2}. In Section \ref{exgen} we prove  Theorem \ref{betgenarr}. In Section \ref{secGCA} we prove Theorems \ref{thmGCb} and \ref{thmRCGA}. In Section \ref{secLR} we relate our results with those of \cite{EiLaMu04} and \cite{Floer}.

\medskip
\noindent
{\it Acknowledgment.} We thank G. Denham, J. Fern\'andez de Bobadilla, H.D. Nguyen, Q.T. L\^e, R. van der Veer, and C. Voll for useful comments. We thank BCAM Bilbao, VIASM Hanoi, and MPI Bonn for the hospitality during the writing of this article. The first author was partly supported by the grant STRT/13/005 from KU Leuven, the Methusalem grant METH/15/026, and the grants G097819N and G0F4216N from FWO. The second author was supported by the Flanders-Vietnam bilateral grant G0F4216N from FWO.

\section{Contact and restricted contact loci}\label{secPre}
Let $X$ be a complex smooth algebraic variety of dimension $n$. One defines the arc space $\LL(X)$ of $X$ to be the scheme parametrizing  all morphisms $\Spec \C\llbracket t \rrbracket \rightarrow X$ over $\bC$, where $\C\llbracket t \rrbracket$ is the ring of formal power series with complex coefficients. For an integer $p \geq 0$, the space $\LL_p(X)$ of $p$-jets of $X$ is the variety parametrizing all morphisms $\Spec \C\llbracket t \rrbracket / (t^{p+1}) \rightarrow X$ over $\bC$. Then $\LL_p(X)$ is a smooth variety of dimension $n(p+1)$, and $\LL(X)$ is the projective limit of $\LL_p(X)$, which is well-defined since the truncation maps $\LL_p(X)\to \LL_{q}(X)$ for all $0\le q\le p$ are affine.

If $f\in \Gamma(X,\O(X))$ is a global regular function, then $\gamma\in \LL(X)$ (or $\LL_p(X)$) gives an element $\gamma(f)$ in $ \C\llbracket t \rrbracket$ (respectively, in  $\C\llbracket t \rrbracket/ (t^{p+1})$). 

\begin{defn}\label{defCL}
For  $p \geq m\in \N$,  the {\it contact loci} in $\LL(X)$, and respectively in $\LL_p(X)$, are:
$$\Cont^m(f):=\{\gamma\in \LL(X) \ |\ \ord_t\gamma(f)=m\},$$
$$\Cont^m(f)_p:=\{\gamma\in \LL_p(X) \ |\ \ord_t\gamma(f)=m\}.$$
The {\it restricted  contact loci} are:
\begin{align*}
&\Cont^{m,1}(f):=\{\gamma\in \LL(X) \ |\ \gamma(f)=t^m+\text{(higher order terms)}\},\\
&\Cont^{m,1}(f)_p:=\{\gamma\in \LL_p(X) \ |\ \gamma(f)=t^m+\text{(higher order terms)}\}.
\end{align*}
We will denote $$\X_m(f):=\Cont^m(f)_m,\quad \XX_m(f):=\Cont^{m,1}(f)_m.$$ 
\end{defn}

From now on, we let $X$ be the affine space $\A^n$ with $n\ge 1$.

An arc $\gamma\in \LL(\A^n)$ (respectively, a jet $\gamma\in \LL_p(\A^n)$) can be identified with its corresponding homomorphism of $\C$-algebras $\C[x_1,\ldots,x_n] \rightarrow \C\llbracket t \rrbracket$ (respectively, to $\C[t]/(t^{p+1})$), which is determined by the images of the variables $x_1,\ldots,x_n$. Let us write
$$\gamma(x_i)=\sum_{j=0}^{\infty} \dfrac{a_{ij}}{j!}t^j.$$
To the polynomial ring $\C[x_1,\ldots,x_n]$ we add new  variables $x_i^{(j)}$ for $j\geq 1$ to form the polynomial ring in infinitely many variables $$S_{\infty}=\C[x_i^{(j)}\ |\ i=1,\ldots,n,\ j\geq 0].$$ One can define a $\C$-derivation $$D:S_{\infty}\rightarrow S_{\infty},\quad \quad D(x_i^{(j)})=x_i^{(j+1)}.$$
Define a map $\C[x_1,\ldots,x_n] \rightarrow \C\llbracket t \rrbracket$ by sending 
$g\mapsto \sum_{j=0}^{\infty} \frac{g^{(j)}(a)}{j!}t^j$, for every $g\in\C[x_1,\ldots,x_n]$, where $g^{(j)}:=D^j(g)$ and $a=(a_{ij}\ |\ 1\leq i\leq n,\ j\geq 0)$ defines $\gamma$ as above. One can  see this is a homomorphism of $\C$-algebras and moreover, its value at the variable $x_i$ coincides with $\gamma(x_i)$. Thus we have
$$
\gamma(f)=\sum_{j=0}^{\infty} \frac{f^{(j)}(a)}{j!}t^j,
$$
and hence
\begin{align*}\Cont^m(f)=  \{a=&(a_{ij})_{1\leq i\leq n,j\geq0}\  |\ a_{ij}\in \C, \text{ and }  \\
 & f(a)=f'(a)=\ldots=f^{(m-1)}(a)=0\ ,\ f^{(m)}(a)\neq 0 \}
 \end{align*}
for $m\ge 0$. 
Using the same argument, one obtains an expression for the truncated version:

\begin{lemma}\label{contru}
\begin{align*}
\Cont^m(f)_p=\{a\in \A^{n\times (p+1)}\ |&\ f(a)=f'(a)=\ldots=f^{(m-1)}(a)=0\\
&\ \text{ and }\ f^{(m)}(a)\neq 0 \}.
\end{align*}
\end{lemma}

Thus we can rewrite the contact locus $\X_m(f)$ and the restricted contact locus $\XX_m(f)$ as:

\begin{lemma}\label{lemXoX}
$$\X_m(f)=\{a\in \A^{n\times (m+1)}\ |\ f(a)=f'(a)=\ldots=f^{(m-1)}(a)=0\ ,\ f^{(m)}(a)\neq 0 \},$$
$$\XX_m(f)=\{a\in \A^{n\times (m+1)}\ |\ f(a)=f'(a)=\ldots=f^{(m-1)}(a)=0\ ,\ f^{(m)}(a)=m!\}.$$
\end{lemma}


\section{Contact loci of hyperplane arrangements}\label{seccont}

Throughout this section, $X=\A^n$ and $\Aa$ is a hyperplane multi-arrangement in $X$ given by a polynomial $f=h_1^{s_1}\ldots h_d^{s_d}$ in $\C[x_1,\ldots,x_n]$, where $n, d, s_i\ge 1$ are integers. Here $h_i$ are polynomials of degree 1, pairwise distinct up to multiplication by a constant, and their zero loci are denoted by $H_i$. 

We do not assume that $\Aa$ is central. If a hyperplane $H$ is given by a non-homogeneous equation $a_1x_1+\ldots +a_nx_n+b=0$, we denote by $H^{center}$ the hyperplane given by the homogeneous part $a_1x_1+\ldots +a_nx_n=0$, and denote by $\Aa^{center}$ the hyperplane arrangement consisting of $H^{center}$ for all $H\in\Aa$. Note that $\Aa^{center}$ has the structure of a hyperplane multi-arrangement by defining the multiplicity of a hyperplane $K\in\Aa^{center}$ to be $\sum_{H^{center}=K}s(H) $. In particular, if $\Aa$ is central, then $H_i^{center}=H_i$ for all $i$, and $\Aa=\Aa^{center}$.

A subset of hyperplanes $S\subset \Aa$ is said to be {\it complete} if the intersection $\cap S$ of all $H\in S$ is non-empty,  and for all $H\in \Aa$, if $H\supset \cap S$ then $H\in S$; the empty set is a complete set by convention. There is a 1-1 correspondence between the complete sets and the intersection lattice of $\Aa$. The multiplicity function $s$ of the multi-arrangement $\Aa$ extends to complete sets by setting $s(S)=\sum_{H\supset S}s(H)$ for a non-empty $S$, and $s(\emptyset)=0$.

Theorem \ref{thmMain} is the central case of the following:

\begin{theorem}\label{thmMain2}
Let $\Aa$ be a hyperplane multi-arrangement in $\A^n$ and $m\in\N$. Then:

(i) The contact locus $\X_m(\Aa)$ admits a disjoint decomposition 
$$\X_{m}(\Aa)=\bigsqcup_{j\in S(m)} \mathscr{C}_{j}(\Aa)$$
where $S(m)=\{j\in\N^d\ |\ j_1s_1+\ldots+j_ds_d=m\}$.

(ii) If $\mathscr{C}_{j}(\Aa)\neq\emptyset$, then $\mathscr{C}_{j}(\Aa)$ is irreducible open and closed in the Zariski topology, and equals the complement of a hyperplane arrangement $\Aa_{j}$ of same rank as $\Aa$ in an affine space $X_j$.

(iii) Define $$T(m)=\{j \in S(m)\mid \mathscr{C}_{j}(\Aa)\neq\emptyset\}.$$ 

For $m=0$, $X_j=X$, $\Aa_j=\Aa$ for $j=(0,\ldots,0)$, so that $\X_0(\Aa)=\mathscr{C}_0(\Aa)=X\setminus \Aa$.

For $m>0$, there is a 1-1 correspondence between $T(m)$ and the set of  descending chains of complete sets in $\Aa$
$$S_0\supset S_1\supset\ldots\supset S_m=\emptyset$$
satisfying $\sum_{k=0}^{m}s(S_k)=m$. The correspondence is
$$
j\mapsto S_\bullet \text{ with } S_k = \{ H_i\in \Aa\mid j_i>k\}\text{ for }0\le k\le m
$$
and, conversely,
$$
S_\bullet \mapsto j\text{ with } j_i=\#\{k\mid H_i\in S_k\} \text{ for }1\le i\le d.
$$
Under this correspondence, 
$$
X_j = Z_0\times Z_1^{center}\times\ldots\times Z_m^{center}
$$
is a product of affine spaces, where $Z_0=\cap_{H\in S_0}H$ and for $1\le k\le m$, 
$$Z_k^{center}=\bigcap_{H\in S_k}H^{center},$$ 
where by convention $Z_k^{center}=\A^{n}$ if $S_k=\emptyset$. The hyperplane arrangement $\Aa_j$ in $X_j$
is the product arrangement
$$
\Aa_j =  \Aa_{j,0}\times\ldots\times \Aa_{j,m}$$
where $\Aa_{j,0}$ is the restriction to $Z_0$ of
$$
\bigcup_{H\in \Aa\setminus S_{0}} H,
$$
and $\Aa_{j,k}$ for $1\le k\le m$ is the restriction to  $Z_k^{center}$ 
of  
$$
\bigcup_{H\in S_{k-1}\setminus S_k} H^{center}.
$$
In particular,
$$
\mathscr{C}_j(\Aa) = X_j\setminus\Aa_j= (Z_0\setminus \Aa_{j,0})\times\times_{k=1}^m (Z_k^{center}\setminus \Aa_{j,k}).
$$

(iv) The combinatorial type of $\Aa$ determines $T(m)$ and the combinatorial type of each $\Aa_j$. The dimension function on the intersection lattice of $\Aa$ determines further the dimension function on the intersection lattice of each $\Aa_j$.
\end{theorem}
\begin{proof} We can focus on the case $m\ge 1$, since the case $m=0$ is obvious.

{\it (i)}  By definition, a homomorphism $\gamma:\C[x_1,\ldots,x_n] \rightarrow \C[t]/(t^{m+1})$ belongs to $\X_m(f)$ if and only if $m=\ord_t\gamma(f)=\sum_{i=1}^d s_i\cdot\ord_t \gamma(h_i)$. Hence
$$\X_{m}(\Aa)=\bigsqcup_{j\in S(m)}\bigcap_{i=1}^d \Cont^{j_i}(h_i)_m$$
where $S(m)=\{(j_1,\ldots,j_d)\in\N^d\ |\ j_1s_1+\ldots+j_ds_d=m\}$.

Denote by $H_i^{(k)}$ the zero set in $\LL_m(X)=\A^{n(m+1)}$ of the $k$-th formal derivative $h_i^{(k)}$ as defined in Section \ref{secPre}, for $k=0,\ldots,m$. By Lemma \ref{contru}, $\Cont^{j_i}(h_i)_m$ can be written as $\bigcap_{k=0}^{j_i-1}H_i^{(k)}\setminus H_i^{(j_i)}$, hence this gives the disjoint decomposition
$$\X_{m}(\Aa)=\bigsqcup_{j\in S(m)}(X_j\setminus \bigcup_{i=1}^d H_i^{(j_i)}),$$
 where $$X_j=\bigcap_{i=1}^d\bigcap_{k=0}^{j_i-1}H_i^{(k)}.$$ The terms in the decomposition are the desired $\mathscr{C}_j(\Aa)$. Note that $j$ determines $m$.

{\it (ii)} Now we prove that $\mathscr{C}_j(\Aa)$ is Zariski open in $\X_{m}(\Aa)$. Let $j$ and $j'$ be different elements of $S(m)$, we then claim that $X_{j'}\setminus \bigcup_{i=1}^d H_i^{(j_i)}=\emptyset$. Indeed, we can choose $i_0=1,\ldots,d$ such that $j'_{i_0}\geq j_{i_0}+1$. Then
$$X_{j'}\setminus \bigcup_{i=1}^d H_i^{(j_i)}\subset H_{i_0}^{(j_{i_0})}\setminus \bigcup_{i=1}^d H_i^{(j_i)}=\emptyset$$
since $H_{i_0}^{(j_i)}$ appears in the intersection $\bigcap_{k=0}^{j'_{i_0}-1}H_{i_0}^{(k)}$. Now for a given $j\in S(m)$, we have
\begin{align*}
\X_m(\Aa)\setminus \bigcup_{i=1}^d H_i^{(j_i)} &= \bigcup_{j'\in S(m)}(X_{j'}\setminus \bigcup_{i'=1}^d H_{i'}^{(j'_{i'})})\setminus \bigcup_{i=1}^d H_i^{(j_i)} \\
&=\bigcup_{j'\in S(m)}((X_{j'}\setminus \bigcup_{i=1}^d H_i^{(j_i)})\setminus\bigcup_{i'=1}^d H_{i'}^{(j'_{i'})})\\
&= (X_j\setminus \bigcup_{i=1}^d H_i^{(j_i)})\setminus\bigcup_{i'=1}^d H_{i'}^{(j_{i'})} \\
& =X_j\setminus \bigcup_{i=1}^d H_i^{(j_i)}.
\end{align*}
which means $\mathscr{C}_j(\Aa)$ is Zariski open in $\X_{m}(f)$.

Because $\mathscr{C}_j(\Aa)$ are open and disjoint, they are also closed in $\X_m(f)$.

For each $j\in S(m)$, $X_j$ is an intersection of $m$ hyperplanes. Thus if $X_j\neq\emptyset$, then it is an affine subspace of $\A^{(m+1)n}$ of codimension at most $m$. The set $X_j\setminus \bigcup_{i=1}^d H_i^{(j_i)}$ is then the complement of the hyperplane arrangement in $X_j$ consisting of non-empty sets $X_j\cap H_i^{(j_i)}$ for all $i\in \{1,\ldots,d\}$. We denote by $\mathscr{A}_j$ this hyperplane arrangement in $X_j$, so that $\mathscr{C}_j(\Aa)=X_j\setminus\Aa_j$ and $\mathscr{C}_j(\Aa)$ must then necessarily be irreducible.


The statement about the rank of $\Aa_j$ will be proved after $(iii)$.

$(iii)$ We claim first that if $\mathscr{C}_j(\Aa)\ne\emptyset$, $X_j$ and $\Aa_j$ admit the product decompositions claimed in $(iii)$.

We can write $$X_j=\bigcap_{k=0}^m \bigcap_{i\in I_k(j)}H_i^{(k)}\quad\text{ and }\quad\bigcup_{i=1}^d H_{i}^{(j_i)}=\bigcup_{k=0}^m\bigcup_{i\in J_k(j)} H_i^{(k)}$$
where \begin{equation}\label{eqIJ}
I_k(j)=\{i\in \{1,\ldots,d\}\mid j_i>k\}\quad\text{ and }\quad J_k(j)=\{i\in \{1,\ldots,d\}\mid j_i=k\}.
\end{equation}
Hence
\begin{equation}\label{eqin}
\mathscr{C}_{j}(\mathscr{A})=\bigcap_{k=0}^m(\bigcap_{i\in I_k(j)}H_i^{(k)}\setminus\bigcup_{i'\in J_k(j)} H_{i'}^{(k)}).
\end{equation}
Note that for any $H\in \Aa$, $H^{(k)}$ is the zero set of a polynomial of variables $x_1^{(k)},\ldots,x_n^{(k)}$, hence it can be considered as a subvariety of $\A^n=\Spec \C[x_1^{(k)},\ldots,x_n^{(k)}]$. This allows us to replace the intersections in  \ref{eqin} by the products,
\begin{equation}\label{conpro}
X_j=\times_{k=0}^m \bigcap_{i\in I_k(j)}H_i^{(k)}\quad\text{ and }\quad \mathscr{C}_{j}(\mathscr{A})=\times_{k=0}^m(\bigcap_{i\in I_k(j)}H_i^{(k)}\setminus\bigcup_{i'\in J_k(j)} H_{i'}^{(k)})
\end{equation}
If $k=0$ then $H^{(0)}=H$, where as if $k>0$, then the isomorphism $\bA^n\rightarrow\bA^n$ identifying $x_1,\ldots, x_n$ with $x_1^{(k)},\ldots,x_n^{(k)}$, respectively, gives an identification $H^{(k)}=H^{center}$. Thus, by setting $Z_k=\bigcap_{i\in I_k(j)}H_i$, and $\mathscr{A}_{j,k}$ the hyperplane arrangement in $\bigcap_{i\in I_k(j)}H_i^{(k)}$ consisting of non-empty sets $\bigcap_{i\in I_k(j)}H_i^{(k)}\cap H_{i'}^{(k)}$ for $i'\in J_k(j)$, we have the product decomposition in $(iii)$.

We are left to prove the following statement: The set $\mathscr{C}_{j}(\Aa)$ is non-empty if and only if the subset $S_k:=\{H_i\ |\ i\in I_k(j)\}$ of $\Aa$ is complete for all $k=0,\ldots,m$.

Let us first prove this for $\Aa$ being a central multi-arrangement. Suppose that $\mathscr{C}_{j}(\Aa)$ is empty. Since we have the decomposition $\mathscr{C}_{j}(\mathscr{A})=\times_{k=0}^m Z_k\setminus\bigcup_{i'\in J_k(j)} H_{i'}^{(k)}$, then there is $k$ such that $$Z_k\setminus\bigcup_{i'\in J_k(j)} H_{i'}=\emptyset,$$ 
In other words, $$Z_k\subset \bigcup_{i'\in J_k(j)} H_{i'},$$ and hence $Z_k\subset H_{i'}$ for some $i'\in J_k(j)$, since $Z_k$ is irreducible. This says that $S_k$ is not complete. 

Conversely, suppose that $S_k$ is not complete. Then there exists $H_{i_0}\supset Z_k$ with $i_0 \notin I_k(j)$. Let $k_0:=j_{i_0}$, we have $k_0\leq k$, hence $I_{k_0}(j)\supset I_{k}(j)$ and hence $Z_{k_0}\subset Z_k$, so that
\begin{align*}
Z_{k_0}\setminus\bigcup_{i'\in J_{k_0}(j)} H_{i'}\subset Z_{k_0}\setminus H_{i_0}\subset Z_k\setminus H_{i_0}=\emptyset
\end{align*}
which says that $\mathscr{C}_{j}(\Aa)$ is empty.

Now let $\Aa$ be an arbitrary multi-arrangement. Note that by the description of $\mathscr{C}_j(\Aa)$ in the proof of $(i)$, we can write the product $\times_{k=1}^m (Z_k^{center}\setminus \Aa_{j,k})$ as
$$\mathscr{C}_{j|_{I_0(j)}-1}(S_0^{center}),$$
where $j|_{I_0(j)}-1=(j_i-1\ |\ i\in I_0(j))$. Thus, $\mathscr{C}_{j}(\Aa)$ is non-empty if and only if both $Z_0\setminus\bigcup_{i'\in J_0(j)} H_{i'}$ and $\mathscr{C}_{j|_{I_0(j)}-1}(S_0^{center})$ are non-empty. It is easy to see that if $\mathscr{C}_{j}(\Aa)\ne 0$ then $S_0$ is complete. Now, with the condition $S_0$ is complete, using the previous part for the central hyperplane multi-arrangement $S_0$, the set $\mathscr{C}_{j|_{I_0(j)}-1}(S_0^{center})$ is non-empty if and only if $S^{center}_1,\ldots,S^{center}_m$ are all complete in $S^{center}_0$. But as $\cap S^{center}_0=Z_0\ne \emptyset$, this is equivalent to saying that $S_1,\ldots,S_m$ are all complete in $S_0$. Since $S_0$ is complete itself in $\Aa$, this is again equivalent to saying that $S_1,\ldots,S_m$ are all complete in $\Aa$. This proves the statement, and also prove the 1-1 correspondence in $(iii)$.

Since complete subsets of $\Aa$ are into 1-1 correspondence with the elements of the intersection lattice of $\Aa$, one obtains the 1-1 correspondence claimed in $(iii)$ of Theorem \ref{thmMain}.

$(ii)-bis$. We prove now that the rank of the hyperplane arrangement $\Aa_j$ equals that of $\Aa$ for $j\in T(m)$.  By Proposition \ref{lemCen} the rank of a hyperplane arrangement does not change under the centralization. Moreover, we have the equality $(\Aa_j)^{center}=(\Aa^{center})_j$. Hence we may assume that $\mathscr{A}$ is central. We have
$$\rank\Aa_j=\sum_{k=0}^m\rank\mathscr{A}_{j,k}=\sum_{k=0}^m\left(\codim\bigcap_{i\in I_k(j)\cup J_k(j)}H_i-\codim\bigcap_{i\in I_k(j)}H_i\right).$$
Note that $I_k(j)\cup J_k(j)=I_{k-1}(j)$, where $I_{-1}(j):=\{1,\ldots,d\}$, thus
\begin{align*}
\rank\Aa_j&=\sum_{k=0}^m\left(\codim\bigcap_{i\in I_{k-1}(j)}H_i-\codim\bigcap_{i\in I_k(j)}H_i\right)\\
&=\codim\bigcap_{i=1}^dH_i-\codim\bigcap_{i\in I_m(j)}H_i\\
&=\rank\Aa,
\end{align*}
since $I_m(j)=\emptyset$.

$(iv)$ This follows from $(iii)$ and Proposition \ref{lemCen}.
\end{proof}


\begin{proposition}\label{lemCen} (i) If $\Aa$ is an affine hyperplane arrangement in $\bA^n$, then the intersection lattice of $\Aa$ determines the intersection lattice of its centralization $\Aa^{center}$ in $\bA^n$.

(ii) The rank of $\Aa$ equals the rank of $\Aa^{center}$.
\end{proposition}
\begin{proof} (i) This is \cite[Theorem 3.2]{WW86}. Note that $\Aa^{center}=(c \Aa)^{H_0}$, where $H_0$ is a new hyperplane ``at infinity" making $\Aa\cup H_0$ a projective arrangement in $\bP^n$, $c \Aa$ is the cone over $\Aa\cup H_0$, and $(c \Aa)^{H_0}$ is the affine arrangement in $H_0$ obtained by restricting the cone, see e.g. \cite[Remark 2.3 and Definition 2.13]{Dimca17}. Then the intersection lattice $L((c \Aa)^{H_0})$ is the interval $[H_0,\hat{1}]$ inside $L(c \Aa)$, and therefore $L((c \Aa)^{H_0})$ is  determined by $L(\Aa)$.

Specifically, the intersection lattice of $\Aa^{center}$ is obtained by identifying the parallel edges in the intersection lattice of $\Aa$. Here two edges $Z_1$ and $Z_2$ of $\Aa$ are said to be parallel if any of the following equivalent conditions holds:
\begin{align*}
Z_1^{center}=Z_2^{center}&\Leftrightarrow \rank Z_1=\rank Z_2 \text{ and } Z_1^{center}\supset Z_2^{center}\\
&\Leftrightarrow\rank Z_1=\rank Z_2 \text{, and } H^{center}\supset Z_2^{center} \text{ for all } H\supset Z_1\\
&\Leftrightarrow\rank Z_1=\rank Z_2 \text{, and } H\supset Z_2 \text{ or } H\cap Z_2=\emptyset \text{ for all } H\supset Z_1.
\end{align*}

(ii) This follows from (i).
\end{proof}
\begin{remark}\label{rmkCC} By counting the number of times a complete set occurs in a chain, we can rephrase the 1-1 correspondence from Theorem \ref{thmMain2} as follows: For $m\ge 1$ there is a 1-1 correspondence between the set of irreducible components of the contact locus $\X_{m}(\Aa)$ and the set of strictly descending chains of complete sets in $\Aa$
$$T_1\supsetneq T_2\supsetneq\ldots\supsetneq T_l\supsetneq\emptyset\quad\text{ for some }l\ge 1$$ together with an assignment $\nu:\{T_1\,\ldots, T_{l}\}\to \N^*$  such that $\sum_{\alpha=1}^{l}\nu(T_{\alpha})s(T_{\alpha})=m$.
\end{remark}


\begin{proof}[Proof of Theorem \ref{thmIsom}]
This is a direct consequence  of Theorem \ref{thmMain} and the Orlik-Solomon theorem.
\end{proof}

\begin{proposition}\label{cohexp} Let $\Aa$ be a hyperplane multi-arrangement in $\A^n$, $m\in \N$, and $R$ an unitary commutative ring. Then the cohomology algebra of the $m$-contact locus of $\Aa$ is the graded algebra
$$
H^*(\X_m(\Aa),R)\simeq \bigoplus_{j\in T(m)} E/ \cI_j.
$$
where  $E=\bigwedge\langle e_1,\ldots,e_d\rangle$, the exterior algebra of the free $R$-module $Re_1\oplus\ldots\oplus Re_d$, and $\cI_j$ is the ideal of $E$ generated by the union of the following sets:

1. The set of $e_J$ such that $J\subset J_0(j)$, and $\bigcap_{i\in I_0(j)\cup J}H_i=\emptyset,$

2. The set of all $\partial e_J$ such that $J\subset J_0(j)$, and $$\rank \bigcap_{i\in I_0(j)\cup J}H_i-\rank\bigcap_{i\in I_0(j)}H_i< |J|,$$

3. The set of all $\partial e_J$ such that $J\subset J_k(j)$ for some $k=1,\ldots,m$, and $$\rank \bigcap_{i\in I_k(j)\cup J}H_i^{center}-\rank \bigcap_{i\in I_k(j)}H_i^{center}< |J|.$$
Here $T(m)$ is defined as in Theorem \ref{thmMain2} {\color{hot}$(iii)$},
and $I_k(j)$, $J_k(j)$ are defined as in (\ref{eqIJ}).

\end{proposition}
\begin{proof}

 We use the notation as in the proof of Theorem \ref{thmMain2}. The direct sum follows from the disjoint decomposition. Next, by definition the Orlik-Solomon algebra of $\Aa_j$ is given by $E/I(\Aa_j)$, where $I(\Aa_j)$ is the ideal of $E$ generated by

i. The set of $e_J$ such that $X_j\cap\bigcap_{i\in J}H_i^{(j_i)}=\emptyset$,

ii. The set of all $\partial e_J$ such that $J$ is dependent in $\Aa_j$, i.e. $$\rank_{\A^{n(m+1)}}(X_j\cap\bigcap_{i\in J}H_i^{(j_i)})-\rank_{\A^{n(m+1)}} X_j<|J|. $$

It is clear that $\cI_j\subset I(\Aa_j)$. Now suppose that $X_j\cap\bigcap_{i\in J}H_i^{(j_i)}=\emptyset$. Since $\mathscr{A}_{j}=\times_{k=0}^m\mathscr{A}_{j,k}$, there is a subset $J'$ of $J$ such that $J'\subset J_k(j)$ for some $k$, and $\bigcap_{i\in I_k(j)\cup J'}H_i^{(k)}=\emptyset$. But this can only happen when $k=0$, because $H_i^{(k)}$ passes through the origin when $k\geq 1$ . This means $e_{J'}\in I$, hence $e_J\in I$. Also, suppose that $J\subset \{1,\ldots,d\}$ is dependent in $\mathscr{A}_j$. Then there is a subset $J'$ of $J$ such that $J'\subset J_k(j)$ for some $k$ and $J'$ is dependent in $\mathscr{A}_{j,k}$, which means that $\partial e_{J'}\in I$. Note that $\partial e_J=\partial e_{J'}.e_{J\setminus J'}+(-1)^{|J'|}e_{J'}.\partial e_{J\setminus J'}$, and $e_{J'}=e_i.\partial e_{J'}$ for any $i\in J'$, so that $\partial e_J\in I$. Thus the proposition holds.
\end{proof}

\section{Restricted contact loci of hyperplane arrangements}\label{secRCL}

In this section we prove Theorem \ref{prorescon2} from the introduction, and its generalization to the non-central case.

Let us first explicit the natural multi-arrangement structure on the the hyperplane arrangements that appeared in Theorem \ref{thmMain2}:

\begin{lemma}\label{lemMAA}
Let $\Aa$ be the hyperplane multi-arrangement in $\bA^n$ defined by $f=h_1^{s_1}\ldots h_d^{s_d}$. With the notation as in Theorem \ref{thmMain2}, let $m\in\N$ and $j\in T(m)$.

(i) The $m$-th formal derivative of $f$ (as defined in Section \ref{secPre}) restricted to the affine space $X_j$ is
$$
f^{(m)}\,_{| X_j} = m! f_j,\quad\text{ with } f_j=\prod_{k=0}^m f_{j,k},\quad\text{ and }
f_{j,k}:=\prod_{i: j_i=k}\left(\frac{h_i^{(k)}\,_{|X_j}}{k!}\right)^{s_i}.
$$

(ii) The polynomial $f_j$ defines the hyperplane multi-arrangement $\Aa_j$.

(iii) The product decomposition $f_j=\prod_{k=0}^m f_{j,k}$ defines the decomposition of hyperplane multi-arrangements $\Aa_j=\times_{k=0}^m\Aa_{j,k}$ for $m\ge 1$.
\end{lemma}
\begin{proof}
First, we write $f=l_1\ldots l_N$ with $N=s_1+\ldots+s_d$, where $l_t$ are the polynomials $h_i$ with repetitions. By induction, one can show that
\begin{equation}\label{eqder}
f^{(m)}=\sum_{\beta_1+\ldots+\beta_N=m} \frac{m!}{\beta_1!\ldots\beta_N!} l_1^{(\beta_1)}\ldots l_N^{(\beta_N)}.
\end{equation}
Since $X_j=\cap_{i=1}^d\cap_{k=0}^{j_i-1}H_i^{(k)}$ where $H_i^{(k)}$ is the zero locus of $h^{(k)}_i$, the restriction of $l_t^{(\beta_t)}$ of $X_j$ is $0$ if $\beta_t<j_i$ for $i$ such that the zero locus $l_t=h_i$. Since $\sum_{i=1}^dj_is_i=m$, this forces
$$
f^{(m)}\,_{|X_j}= \frac{m!}{(j_1!)^{s_1}\ldots(j_d!)^{s_d}}\prod_{i=1}^dh_i^{(j_i)}\,_{|X_j}.
$$
This gives $(i)$. Parts $(ii)$ and $(iii)$ follow immediately from $(i)$ and the proof of Theorem \ref{thmMain2}.
\end{proof}

 We have a decomposition for the restricted contact loci as follows:

\begin{theorem}\label{prorescon} Let $\Aa$ be a hyperplane multi-arrangement in $\bA^n$ defined by $f=h_1^{s_1}\ldots h_d^{s_d}$ as in Theorem \ref{thmMain2}. For $m\in \N$, there is a disjoint decomposition into Zariski open and closed subsets of the restricted contact locus  
$$\XX_{m}(f)=\bigsqcup_{j\in T(m)}\mathscr{F}_j(f),$$
where $\mathscr{F}_j(f)$ is the fiber of $f_j$ at the point $(j_1!)^{s_1}\ldots(j_d!)^{s_d}\in \C$, with $f_j$ as in Lemma \ref{lemMAA}. 

In particular, if $\Aa$ is central, $\mathscr{F}_j(f)$ is the Milnor fiber of $\Aa_j$ for every $j\in T(m)$.
\end{theorem}
\begin{proof} Using the description of restricted contact locus $\XX_m(f)$ from Lemma \ref{lemXoX} and the decomposition from Theorem \ref{thmMain2},
$$
\XX_m(f) = \X_m(f)\cap Z(f^{(m)}-m!)= \bigsqcup_{j\in T(m)}(X_j\setminus \Aa_j)\cap Z(f^{(m)}-m!),
$$
where $Z(f)$ denotes the zero locus of $f$. Then
$$
\XX_m(f)=\bigsqcup_{j\in T(m)} Z(f_j-(j_1!)^{s_1}\ldots(j_d!)^{s_d})
$$
by Lemma \ref{lemMAA} $(ii)$. This proves the first claim.

If $f$ is central case, then $f_j$ are also central, and the  fiber over any non-zero point of a polynomial defining a central hyperplane multi-arrangement is its Milnor fiber.
\end{proof}


\begin{rmk}
It is natural to expect that the Betti numbers of the restricted contact loci of hyperplane arrangements are combinatorial invariants. The similar question for the Milnor fibers of central hyperplane arrangements is a well-known open problem.
\end{rmk}

\section{Generic hyperplane arrangements}\label{exgen} 

In this section we compute the Betti numbers for contact loci of {generic} hyperplane arrangements, proving Theorem \ref{betgenarr}. Recall that an arrangement $\B^n_d$ of $d\ge 1$ hyperplanes  in $\A^n$ is said to be {\it generic}  if it is reduced and if for any subset of hyperplanes $S\subset\B^n_d$, we have $\codim \bigcap S=|S|$ when $|S|\leq n$, and $\bigcap S=\emptyset$ when $|S|>n$. 

We will use the notation $M(\Aa)$ for the complement of a hyperplane arrangement $\Aa$ in general. 

We need to use the following well-known result on Betti polynomials of complements of generic arrangements, see \cite[Example 2.19]{Dimca17}:

\begin{lemma}\label{genarr}
If $\B^n_d$ is a generic  arrangement of $d\ge 1$ hyperplanes in $\A^n$, then:

1. $B(M(\B^n_d),t)=(1+t)^d$ if $d\leq n$,

2. $B(M(\B^n_d),t)=\sum_{k=0}^n\binom{d}{k}t^k$ if $d>n$.
\end{lemma}

\begin{proof}[Proof of Theorem \ref{betgenarr}] We will use Theorem \ref{thmMain2} and the notation introduced in its proof, with $\Aa$ now being $\B^n_d$.

First, assume that $d\leq n$. In this case, every subset of $\B^n_d$ is complete, hence $\mathscr{C}_{j}(\B^n_d)$ is non-empty for any $j\in S(m)$. Moreover, every $(\B^n_d)_{j}$ is also a generic arrangement. We have  that
\begin{equation}\label{d<n}
B(\X_{m}(\B^n_d),t)=|S(m)|\cdot(1+t)^d=\begin{pmatrix}m+d-1\\m\end{pmatrix}(1+t)^d.
\end{equation}
We will use this expression in what follows.

Now, let $d>n$. Note that a set $S\subset\B^n_d$ is complete if and only if $|S|\leq n$. Hence $\mathscr{C}_{j}(\B^n_d)$ is nonempty if and only if $I_k(j)\leq n$, for all $k=0,\ldots,m$, which is equivalent to say that $I_0(j)\leq n$. Thus
$$\X_{m}(\B^n_d)=\bigsqcup_{j\in S(m),\; |I_0(j)|\leq n}\mathscr{C}_{j}(\B^n_d)=\bigsqcup_{l=1}^n\bigsqcup_{j\in S(m),\; |I_0(j)|=l}\mathscr{C}_{j}(\B^n_d).$$
Write $(\B^n_d)_{j}$ under the form $(\B^n_d)_{j,0}\times (\B^n_d)_{j,\geq 1}$, we then have
$$B(\X_{m}(\B^n_d),t)=\sum_{l=1}^n\sum_{|I_0(j)|=l}B(M((\B^n_d)_{j,0}),t)\cdot B(M((\B^n_d)_{j,\geq1}),t).$$
Note that if $|I_0(j)|=l$, $(\B^n_d)_{j,0}$ is a generic arrangement with $d-l$ hyperplanes in a $(n-l)$-dimensional affine space, hence $$B(M((\B^n_d)_{j,0}),t)=\sum_{i=0}^{n-l}\begin{pmatrix}d-l\\i\end{pmatrix}t^i$$ by Lemma \ref{genarr}. Meanwhile, $(\B^n_d)_{j,\geq1}$ is a generic arrangement with $l$ hyperplanes in a $(nm-m+l)$-dimensional space, with $l\leq nm-m+l$, hence $$B(M((\B^n_d)_{j,\geq1}),t)=(1+t)^l.$$ Thus
\begin{align*}
B(\X_{m}(\B^n_d),t)&=\sum_{l=1}^n\left(\sum_{i=0}^{n-l}\begin{pmatrix}d-l\\i\end{pmatrix}t^i\right)(1+t)^l\sum_{j\in S(m),\; |I_0(j)|=l}1\\
&=\sum_{l=1}^n\begin{pmatrix}d\\l\end{pmatrix}\begin{pmatrix}m-1\\m-l\end{pmatrix}\left(\sum_{i=0}^{n-l}\begin{pmatrix}d-l\\i\end{pmatrix}t^i\right)(1+t)^l.
\end{align*}
In particular, the $k$-th Betti number is
\begin{equation}\label{eqBsB}
b_k(\X_m(\B^n_d))=\sum_{l=1}^n\begin{pmatrix}d\\l\end{pmatrix}\begin{pmatrix}m-1\\m-l\end{pmatrix}\sum_{i} \begin{pmatrix}d-l\\i\end{pmatrix}\begin{pmatrix}l\\k-i\end{pmatrix},
\end{equation}
for every $k\in [n]$, where the sum $\sum_{i}$ is over all $i$ satisfying $0\leq i\leq n-l$ and $0\leq k-i\leq l$. We will reduce this to obtain a nicer fomula of Betti numbers. First, in the sum indexed by $i$, we can substitute $i$ by $n-l-i$. We get

\begin{align*}
b_k(\X_m(\B^n_d))&=\sum_{l=1}^n\begin{pmatrix}d\\l\end{pmatrix} \begin{pmatrix}m-1\\m-l\end{pmatrix} \sum_{i} \begin{pmatrix}d-l\\n-i-l\end{pmatrix} \begin{pmatrix}l\\n-k-i\end{pmatrix}\\
&= \sum_{l=1}^n\sum_{i}\begin{pmatrix}d\\l\end{pmatrix} \begin{pmatrix}d-l\\n-i-l\end{pmatrix}  \begin{pmatrix}l\\n-k-i\end{pmatrix} \begin{pmatrix}m-1\\m-l\end{pmatrix},
\end{align*}
where $\sum_{i}$ now is over all $i$ such that $0\leq i\leq n-l$ and $0\leq n-k-i\leq l$. Using the well-known formula
\begin{equation}\label{comfor}
\begin{pmatrix}a\\b\end{pmatrix}\begin{pmatrix}b\\c\end{pmatrix}=\begin{pmatrix}a\\c\end{pmatrix}\begin{pmatrix}a-c\\b-c\end{pmatrix}
\end{equation}
for any natural number $a,b,c$ such that $a\geq b\geq c$, we have
\begin{align*}
\begin{pmatrix}d\\l\end{pmatrix} \begin{pmatrix}d-l\\n-i-l\end{pmatrix}\begin{pmatrix}l\\n-k-i\end{pmatrix}&=
\begin{pmatrix}d\\n-i\end{pmatrix} \begin{pmatrix}n-i\\l\end{pmatrix}\begin{pmatrix}l\\n-k-i\end{pmatrix}\\
&=
\begin{pmatrix}d\\n-i\end{pmatrix} \begin{pmatrix}n-i\\k\end{pmatrix}\begin{pmatrix}k\\l+k+i-n\end{pmatrix}\\
&=\begin{pmatrix}d\\k\end{pmatrix} \begin{pmatrix}d-k\\n-i-k\end{pmatrix}\begin{pmatrix}k\\l+k+i-n\end{pmatrix}
\end{align*}
Substitute this into $b_k$, and change the order of the sums indexed by $l$ and $i$. We get
$$b_k(\X_m(\B^n_d))=\begin{pmatrix}d\\k\end{pmatrix}\sum_{i=0}^{n-k} \begin{pmatrix}d-k\\n-i-k\end{pmatrix} \sum_l \begin{pmatrix}k\\l+k+i-n\end{pmatrix}\begin{pmatrix}m-1\\m-l\end{pmatrix},$$
where $\sum_l$ is over all $l$ such that $n-i-k\leq l\leq n-i$. Notice that the sum over $l$ in the last display is the coefficient of $t^{m+k+i-n}$ in the polynomial $(1+t)^{m+k-1}=(1+t)^k(1+t)^{m-1}$, which is equal to $\binom{m+k-1}{m+k+i-n}$. Substitute $i$ by $n-k-i$ again, we get that
\begin{equation}\label{genarr1}
b_k(\X_m(\B^n_d))=\begin{pmatrix}d\\k\end{pmatrix}\sum_{i=0}^{n-k} \begin{pmatrix}d-k\\i\end{pmatrix}\begin{pmatrix}m+k-1\\m-i\end{pmatrix}.
\end{equation}
as desired. 

Remark that (\ref{genarr1}) still holds when $d\leq n$. Indeed, under the assumption $d\leq n$, (\ref{genarr1}) becomes
$$b_k(\X_m(\B^n_d))=\begin{pmatrix}d\\k\end{pmatrix}\sum_{i=0}^{d-k} \begin{pmatrix}d-k\\i\end{pmatrix}\begin{pmatrix}m+k-1\\m-i\end{pmatrix}$$
because $\binom{d-k}{i}=0$ when $i>d-k$. The sum in the last display is exactly the coefficient of $t^m$ in the polynomial $(1+t)^{m+d-1}=(1+t)^{d-k}(1+t)^{m+k-1}$. Hence $$b_k(\X_m(\B^n_d))=\begin{pmatrix}d\\k\end{pmatrix}\begin{pmatrix}m+d-1\\m\end{pmatrix},$$ which agrees with (\ref{d<n}). This finishes the proof of the theorem.
\end{proof}

\section{Generic central hyperplane arrangements}\label{secGCA}

In this section we compute the Betti numbers for contact loci and restricted contact loci of generic central hyperplane arrangements. 

Recall that a hyperplane arrangement $\G^n_d$ in $\A^n$ of $d\ge 1$ hyperplanes is called {\it generic central} if it satisfies the following condition: for any subset of hyperplanes $S\subset \G^n_d$, we have $\codim\cap S=|S|$ if $|S|\leq n$ and $\cap S=\{0\}$ if $|S|>n$. If $d\leq n$, this returns to the notion of generic arrangement.

\begin{lemma}\label{lemgencen}
If $\G^n_d$ is a generic central hyperplane arrangement, the complement admits an isomorphism
$$M(\G^n_d)\simeq M(\B^{n-1}_{d-1})\times\C^*.$$
In particular, the Betti polynomial is
\begin{align*}B(M(\G^n_d),t)&=(1+t)B(M(\B^{n-1}_{d-1}),t) =(1+t)\sum_{k=0}^{n-1}\begin{pmatrix}d-1\\k\end{pmatrix}t^k\\
&=
\sum_{k=0}^{n-1}\begin{pmatrix}d\\k\end{pmatrix}t^k+\begin{pmatrix}d-1\\n-1\end{pmatrix}t^n.
\end{align*}
\end{lemma}
\begin{proof}
It is easy to see that any decone $d\G^n_d$ of $\G^n_d$ is a generic hyperplane arrangement in $\A^{n-1}$ consisting of $d-1$ hyperplanes, see \cite[Remark 2.3]{Dimca17}. Moreover, we have an isomorphism
$$M(\G^n_d)\xrightarrow{\simeq}M(d\G^n_d)\times\C^*$$
by \cite[Proposition 2.1]{Dimca17}. The rest of the conclusion follows from Lemma \ref{genarr}.
\end{proof}

\begin{lemma}\label{progencen}
Let $\G^n_d$ be a generic central hyperplane arrangement with $d>n$, and $m\in \N^*$. 

\begin{enumerate}
\item Assume that $m=pd+q$ is the Euclidean division of $m$ by $d$, then $$\X_m(\G^n_d)\simeq\bigsqcup_{\alpha=0}^p\; \bigsqcup_{j\in T(m-\alpha d),\; |I_0(j)|<n}\mathscr{C}_j(\G^n_d)$$
as isomorphism of algebraic varieties, with $I_0(j)$ defined as in (\ref{eqIJ}).

\item The $k$-th Betti number of $\bigsqcup_{j\in T(m),\; |I_0(j)|<n}\mathscr{C}_j(\G^n_d)$ is
$$\begin{pmatrix}d\\k\end{pmatrix}\sum_{i=0}^{n-1-k}\begin{pmatrix}d-k\\i\end{pmatrix}\begin{pmatrix}m+k-1\\m-i\end{pmatrix}+\sum_{l=1}^{n-1}\begin{pmatrix}d\\l\end{pmatrix}\begin{pmatrix}m-1\\m-l\end{pmatrix}\begin{pmatrix}d-l-1\\n-l-1\end{pmatrix}\begin{pmatrix}l\\k+l-n\end{pmatrix}.$$
\end{enumerate}
\end{lemma}

\begin{proof}
(1) Note that a set $S\subset \G^n_d$ is complete if and only if $|S|<n$, or $|S|=d$. Hence
$$\X_m(\G^n_d)=\bigsqcup_{j\in T(m),\; |I_0(j)|<n}\mathscr{C}_j(\G^n_d)\sqcup\bigsqcup_{j\in T(m),\; |I_0(j)|=d}\mathscr{C}_j(\G^n_d),$$
by the proof of Theorem \ref{thmMain2}. Note that second disjoint union is isomorphic to $\X_{m-d}(\G^n_d)$. Then we obtain the claim by induction.

(2) Let us write
$$\bigsqcup_{j\in T(m),\; |I_0(j)|<n}\mathscr{C}_j(\G^n_d)=\bigsqcup_{l=1}^{n-1}\;\bigsqcup_{j\in T(m),\; |I_0(j)|=l}\mathscr{C}_j(\G^n_d).$$
We then have
$$B(\bigsqcup_{j\in T(m),\; |I_0(j)|<n}\mathscr{C}_j(\G^n_d),t)=\sum_{l=1}^{n-1}\sum_{|I_0(j)|=l}B(M((\G^n_d)_{j,0}),t)\cdot B(M((\G^n_d)_{j,\geq1}),t)$$
by the product description of $\mathscr{C}_j(\G^n_d)$ from Theorem \ref{thmMain} (iii).

When $|I_0(j)|=l$, we have $(\G^n_d)_{j,0}=\G^{n-l}_{d-l}$, and $(\G^n_d)_{j,\geq1}=\G^{nm-m+l}_l$.

Now we can use Lemma \ref{lemgencen} and obtain that 
\begin{align}\label{eqNeed}
B( & \bigsqcup_{j\in T(m),\;  |I_0(j)|<n}\mathscr{C}_j(\G^n_d),t) = \nonumber\\
&=\sum_{l=1}^{n-1}\begin{pmatrix}d\\l\end{pmatrix}\begin{pmatrix}m-1\\m-l\end{pmatrix}\left(\sum_{i=0}^{n-1-l}\begin{pmatrix}d-1-l\\i\end{pmatrix}t^i\right)(1+t)^{l+1}\\
&=\sum_{l=1}^{n-1}\begin{pmatrix}d\\l\end{pmatrix}\begin{pmatrix}m-1\\m-l\end{pmatrix}\left(\sum_{i=0}^{n-1-l}\begin{pmatrix}d-l\\i\end{pmatrix}t^i+\begin{pmatrix}d-l-1\\n-l-1\end{pmatrix}t^{n-l}\right)(1+t)^l \nonumber
\end{align}

In particular, the coefficient of $t^k$ is 
$$
\sum_{l=1}^{n-1}\begin{pmatrix}d\\l\end{pmatrix}\begin{pmatrix}m-1\\m-l\end{pmatrix}\sum_{i=0}^{n-1-l}\begin{pmatrix}d-l\\i\end{pmatrix}\begin{pmatrix}l\\k-i\end{pmatrix}+\sum_{l=1}^{n-1}\begin{pmatrix}d\\l\end{pmatrix}\begin{pmatrix}m-1\\m-l\end{pmatrix}\begin{pmatrix}d-l-1\\n-l-1\end{pmatrix}\begin{pmatrix}l\\k+l-n\end{pmatrix}.
$$
The first sum is exactly the $k$-th Betti number of $\X_m(\B_d^{n-1})$ by (\ref{eqBsB}), and can be therefore simplified as in Theorem \ref{betgenarr}, as claimed. 
\end{proof}

\begin{proof}[Proof of Theorem \ref{thmGCb}] In the case that $d$ does not divides $m$,
Lemma \ref{progencen} gives directly that
$$
b_k(\X_m(\Aa)) =\sum_{j=0}^{\lfloor \frac{m}{d}\rfloor}\sum_{l=1}^{n-1}\begin{pmatrix}d\\l\end{pmatrix}\begin{pmatrix}m-jd-1\\l-1\end{pmatrix}\sum_{i=0}^{n-1-l}\begin{pmatrix}d-1-l\\i\end{pmatrix}\begin{pmatrix}l\\k-i\end{pmatrix}.
$$
In the case that $d$ divides $m$, this summation only runs for $0\le j\le \frac{m}{d}-1$ and one has to add $b_k(M(\Aa))$, which is given by Lemma \ref{lemgencen}. The statement now follows  easily.
\end{proof}

Next we address the Betti numbers of the restricted contact loci of $\mathcal{G}^n_d$. We will use the notation $F(\Aa)$ for the Milnor fiber at the origin of a central hyperplane multi-arrangement $\Aa$. Recall that if $f$ is a defining polynomial for $\Aa$, then $F(\Aa)\simeq\{f=c\}$ for every $c\in\bC^*$. We need a few preliminary results.

\begin{lemma}\label{lemmilfib}
Let $n, d, m\in \N^*$ with $n\geq d$.
\begin{enumerate}
\item For any central hyperplane arrangement $\Aa$, and any two hyperplane arrangements $\Bb$, $\Bb'$ of the type $\G^n_d$, the Milnor fibers $F(\Aa\times \Bb)$ and $F(\Aa\times \Bb')$ are isomorphic as algebraic varieties. Moreover,
$$F(\Aa\times \G^n_d)\simeq M(\Aa\times \G^{n-1}_{d-1}).$$
In particular, $F(\G^n_d)\simeq M(\G^{n-1}_{d-1})$ and  $M(\Aa\times \G^n_d)\simeq F(\Aa\times \G^n_d)\times \C^*$.

\item The Betti polynomial of the restricted $m$-contact locus of $\G^n_d$ is
$$B(\XX_m(\G^n_d),t)=\begin{pmatrix}m+d-1\\m\end{pmatrix}(1+t)^{d-1}.$$
\end{enumerate}
\end{lemma}

\begin{proof}
(1) We can assume that $\Bb'$ is the hyperplane arrangement in $\A^n$ with coordinates $y_1,\ldots, y_n$ defined by $y_1\ldots y_d=0$. Assume that the defining polynomial of $\Bb$ is $g=g_1\ldots g_d$. The genericity of $\Bb$ tells us that the linear polynomials $g_1,\ldots,g_d$ are linearly independent. Thus we can find an isomorphism of $\A^n$ sending $y_i$ to $g_i$ for all $i=1,\ldots, d$. Using this isomorphism, one maps isomorphically the set $F(\Aa\times \Bb)$ to $F(\Aa\times \Bb')$.

Now assume that $f\in\C[x_1,\ldots,x_N]$ is the defining polynomial of $\Aa$, and we can assume that $\G^n_d$ is the hyperplane arrangement defined by $y_1\ldots y_d=0$. Then $F(\Aa\times\G^n_d)$ is the subvariety of $\A^N\times\A^n$ defined by $f\cdot y_1\ldots y_d=1$. The map
\begin{align*}
u:F(\Aa\times\G^n_d)&\rightarrow M(\Aa\times \G^{n-1}_{d-1})\\
(a,b)&\mapsto(a,b_1,\ldots,\widehat{b_d},\ldots,b_n)
\end{align*}
is an isomorphism with inverse
$$(a,c)\mapsto (a,c_1,\ldots,c_{d-1},\frac{1}{f(a)c_1\ldots c_{d-1}},c_d,\ldots,c_{n-1}).$$
Finally, since $\G^n_d=\G^{n-1}_{d-1}\times \G^1_1$ when $n\ge d\ge 1$, we get
$$M(\Aa\times \G^{n}_{d})= M(\Aa\times \G^{n-1}_{d-1})\times \C^*\simeq F(\Aa\times \G^n_d)\times \C^*.$$

(2) By Theorem \ref{prorescon}, we have a decomposition
$$\XX_m(\G^n_d)=\bigsqcup_{j\in T(m)}\mathscr{F}_j(\G^n_d)$$
where $\mathscr{F}_j(\G^n_d)$ is the Milnor fiber of $(\G^n_d)_j$, a generic arrangement. Hence,
$\mathscr{F}_j(\G^n_d)\simeq F(\G^n_d)\simeq M(\G^{n-1}_{d-1})$
by (1). Moreover, in this case $T(m)=S(m)$ and therefore $B(\XX_m(\G^n_d),t)=\binom{m+d-1}{m}(1+t)^{d-1}$ .
\end{proof}

We will also need the following result of \cite[Theorem 2.6]{OR93}: 

\begin{theorem}\label{lembetmil} Let $\G^n_d$ be a generic central  arrangement of $d$ hyperplanes in $\bA^n$ with $d>n\geq 2$, then for $0\leq k\leq n-2$,
$$b_k(F(\G^n_d))=\begin{pmatrix}d-1\\k\end{pmatrix},$$
and
$$b_{n-1}(F(\G^n_d))=\begin{pmatrix}d-2\\n-2\end{pmatrix}+d\begin{pmatrix}d-2\\n-1\end{pmatrix}.$$
\end{theorem}

\begin{lemma}\label{propex3} Let $n, d, m\in \N^*$ with $d>n$. Let $\G^n_d$ be a generic central hyperplane arrangement. 
\begin{enumerate}
\item Assume that $m=pd+q$ is the Euclidean division of $m$ by $d$, then 
$$\XX_m(\G^n_d) \simeq \bigsqcup_{\alpha=0}^p\; \bigsqcup_{j\in T(m-\alpha d),\; |I_0(j)|<n}\mathscr{F}_j(\G^n_d)$$
as isomorphism of algebraic varieties.

\item We have the following isomorphism of algebraic varieties:
$$\bigsqcup_{j\in T(m),\; |I_0(j)|<n}\mathscr{C}_j(\G^n_d)\simeq\bigsqcup_{j\in T(m),\; |I_0(j)|<n}\mathscr{F}_j(\G^n_d)\times\C^*.$$
In particular, the $k$-th Betti number of $\bigsqcup_{j\in T(m),\; |I_0(j)|<n}\mathscr{F}_j(\G^n_d)$ is
\begin{equation}\label{eqbetres}
\sum_{l=1}^{n-1}\begin{pmatrix}d\\l\end{pmatrix}\begin{pmatrix}m-1\\m-l\end{pmatrix}\sum_{i=0}^{n-1-l}\begin{pmatrix}d-1-l\\i\end{pmatrix}\begin{pmatrix}l\\k-i\end{pmatrix}.
\end{equation}

\item If $d$ does not divide $m$, then $\X_m(\G^n_d)\simeq\XX_m(\G^n_d)\times \C^*$.

\end{enumerate}
\end{lemma}
\begin{proof}
(1) Follows from the decomposition in Lemma \ref{progencen}.


(2) As with previous examples, we want to write
$$\bigsqcup_{j\in T(m),\; |I_0(j)|<n}\mathscr{F}_j(\G^n_d)=\bigsqcup_{l=1}^{n-1}\bigsqcup_{j\in T(m),\; |I_0(j)|=l}\mathscr{F}_j(\G^n_d)=\bigsqcup_{l=1}^{n-1}\bigsqcup_{j\in T(m),\; |I_0(j)|=l}F((\G^n_d)_j),$$
and write $(\G^n_d)_j=(\G^n_d)_{j,0}\times(\G^n_d)_{j,\geq1}$. Since $(\G^n_d)_{j,\geq1}=\G^{nm-m+l}_l$ and $l\leq nm-m-l$, Lemma \ref{lemmilfib} implies that
$$\mathscr{C}_j(\G^n_d)=M((\G^n_d)_j)\simeq F((\G^n_d)_j)\times \C^*.$$
Thus we obtain (\ref{eqbetres}) from (\ref{eqNeed}).

(3) This follows from (1) and (2).
\end{proof}

\begin{proof}[Proof of Theorem \ref{thmRCGA}] If $d$ does not divide $m$,  Lemma \ref{propex3} gives directly that
$$
b_k(\XX_m(\Aa)) =\sum_{j=0}^{\lfloor \frac{m}{d}\rfloor}\sum_{l=1}^{n-1}\begin{pmatrix}d\\l\end{pmatrix}\begin{pmatrix}m-jd-1\\l-1\end{pmatrix}\sum_{i=0}^{n-1-l}\begin{pmatrix}d-1-l\\i\end{pmatrix}\begin{pmatrix}l\\k-i\end{pmatrix}.
$$
If $d$ divides $m$, the first summation only runs for $0\le j\le \frac{m}{d}-1$ and one has to add $b_k(F(\Aa))$, which is given by Lemma \ref{lembetmil}. The statement now follows  easily.
\end{proof}

\section{Log resolutions}\label{secLR}
In this section, we compare Theorem \ref{thmMain2} with the general description of contact loci from \cite{EiLaMu04} in terms of log resolutions. We use this comparison to prove the degeneracy of a spectral sequence from \cite{Floer} associated to an $m$-separating log resolution, relating conjecturally the restricted $m$-contact locus with the Floer cohomology of the $m$-iterate of the Milnor monodromy.

We consider first more generally a non-invertible regular function $f:X\rightarrow\bC$ on a smooth complex variety. Fix a log resolution $$\varphi:Y\rightarrow X$$ of $f$. Write
$$(f\circ\varphi)^{-1}(0)=\sum_{i=1}^sN_iE_i,$$ where  $E=\cup_{i=1}^sE_i$ is a simple normal crossings divisor, $E_i$ are the mutually distinct irreducible components of $E$, and $N_i\geq 0$. Write the relative canonical divisor as
$$K_{Y/X}:=K_{Y}-\varphi^*K_X=\sum_{i=1}^sk_iE_i.$$
Note that $\varphi$ induces the morphisms $\varphi_\infty:\mathcal{L}(Y)\to \mathcal{L}(X)$ and $\varphi_p:\mathcal{L}_p(Y)\to \mathcal{L}_p(X)$ for every $p\in\N$. For  $\nu\in \N^s$, one defines the multi-contact locus
$$\Cont^{\nu}(E)=\{\gamma\in \mathcal{L}(Y) |\ \ord_{\gamma}(E_i)=\nu_i\hbox{ for all }1\leq i\leq s\}$$
and similarly, if $\nu_i\le p$ for all $1\le i\le s$,
$$\Cont^{\nu}(E)_p=\{\gamma\in \mathcal{L}_p(Y) |\ \ord_{\gamma}(E_i)=\nu_i\hbox{ for all }1\leq i\leq s\}.$$

\begin{theorem}{\rm(}\cite[Theorem A]{EiLaMu04}{\rm)}\label{eilamu} For a log resolution $\varphi:Y\rightarrow X$ of a regular function $f:X\rightarrow \C$ on a smooth complex variety $X$, and for any positive integer $m$, there is a disjoint union decomposition 
\be\label{eqELM}\Cont^m(f)=\bigsqcup_{\sum_i \nu_iN_i=m}\varphi_{\infty}(\Cont^{\nu}(E)).
\ee
Each $\varphi_{\infty}(\Cont^{\nu}(E))$ is an irreducible constructible cylinder of codimension $\sum_i\nu_i(k_i+1)$. For every irreducible component $W$ of $\Cont^m(f)$ there is a unique $\nu$ such that $\Cont^{\nu}(E)$ dominates $W$.
\end{theorem}

Recall that for hyperplane arrangements, one has a canonical log resolution by blowing up inductively by increasing dimension (the strict transforms of) all edges of codimension 2 or higher. In this case, Theorem \ref{thmMain2} implies that the stratification (\ref{eqELM}) is as nice as possible if $\varphi$ is the canonical log resolution:

\begin{proposition}\label{propCompELM}
If $f:X=\C^n\rightarrow\C$ is a hyperplane multi-arrangement and $\varphi:Y\ra X$ is the canonical log resolution of $f$, then (\ref{eqELM}) is the decomposition into connected components and each component is irreducible, after removing the unnecessary terms with $\Cont^\nu(E)=\emptyset$. 

More precisely, there is a 1-1 correspondence between
$$
\{\nu\in\N^s \mid \sum_i\nu_iN_i=m\text{ and } \Cont^\nu(E)\neq\emptyset\}
$$
and $T(m)$ defined as in Theorem \ref{thmMain2}, identifying a non-empty a set $\varphi_{\infty}(\Cont^{\nu}(E))$ uniquely with $\pi_m^{-1}(\mathscr{C}_j(\Aa))$ for some $j\in T(m)$, where $\pi_m:\mathcal{L}(X)\ra\mathcal{L}_m(X)$ is the truncation map.
\end{proposition}
\begin{proof}
By Theorem \ref{thmMain2}, 
$$
\Cont^m(f)=\bigsqcup_{j\in T(m)}\pi_m^{-1}(\mathscr{C}_j(\Aa))
$$
is the decomposition into connected components and each component is irreducible. By the last statement in Theorem \ref{eilamu}, it suffices then that we prove the claimed 1-1 correspondence. 

This follows immediately from the description of the strata of the canonical log resolution. Indeed, in this case the irreducible components of $E$ are prime divisors $E_Z$ indexed by the edges $Z\neq \A^n$ of $\Aa$ and the order of vanishing of $f$ along $E_Z$ is $N_Z=s(Z)$. For a subset $\mathcal{F}\subset L(\Aa)\setminus\{\A^n\}$, the intersection 
$$
E_\mathcal{F}=\bigcap_{Z\in \mathcal{F}}E_Z \neq\emptyset
$$
if and only if $\mathcal{F}$ is nested, that is, $\mathcal{F}$ consists of a chain of inclusions. 

Take $\nu=(\nu_Z)$ such that $\sum_Z\nu_ZN_Z=m$ and $\Cont^\nu(E)\neq \emptyset$. This is equivalent to $\sum_Z\nu_Zs(Z)=m$ and $E_\mathcal F\neq\emptyset$, where $\mathcal F=\{Z\mid \nu_Z\neq 0\}$. Hence $\mathcal F$ is a nested set of edges not equal to $\A^n$. Take each edge $Z$ in $\mathcal F$ exactly $\nu_Z$ times to obtain an increasing chain of edges appearing multiple times, and augment this chain by $\A^n$'s to obtain a chain 
$$
Z_0\subset Z_1\subset\ldots \subset Z_m=\A^n
$$
of precisely $m+1$ edges, with repetitions, such that $\sum_{k=0}^ms(Z_k)=m$. Define $S_k=\{H\in\Aa\mid Z_k\subset H\}$, and hence  obtain a chain of complete sets in $\Aa$
$$
S_0\supset S_1\supset \ldots S_m=\emptyset
$$
satisfying that $\sum_{k=0}^ms(S_k)=m$, that is, an element of $T(m)$, according to Theorem \ref{thmMain2} (iii). Since this process can be reversed, we have the claimed 1-1 correspondence.
\end{proof}

Returning for the moment to the general case of a non-invertible regular function $f:X\ra\C$ on a smooth complex variety, we say that a log resolution $\varphi:Y\ra X$ of $f$ is {\it $m$-separating} if $m_i+m_j>m$ for all $i\ne j$ with $E_i\cap E_j\neq\emptyset$. Let $W=-\sum_iw_iE_i$ be a relatively ample divisor, with $w_i\ge 0$ for all $i$, and $w_i=0$ if $E_i$ is not an exceptional divisor. 
Let $S_m=\{i\mid N_i\text{ divides }m\}$,  $S_{m,p}=\{i\in S_m\mid mw_i/N_i=-p\}$,  $E_i^\circ=E_i\setminus \cup_{j\ne i}E_j$, and $\tilde{E}_i^\circ\ra E_i^\circ$ the unramified cyclic cover of degree $m$ associated to $f$, see \cite{Floer}. This type of log resolutions are useful for computing cohomology with compact supports of the restricted $m$-contact locus $\XX_m(f)$ of $f$ in  $\mathcal L_m(X)$:

\begin{theorem}{\rm(}\cite[Theorem 1.1]{Floer}{\rm)}\label{thmFloer} For a positive integer $m$ and an $m$-separating log resolution $\varphi:Y\rightarrow X$ of a regular function $f:X\rightarrow \C$ on a smooth complex variety $X$ of dimension $n$, there is a cohomological spectral sequence
\be\label{eqSS}
E_1^{p,q}=\bigoplus_{i\in S_{m,p}} H_{2(n(m+1)-m\frac{k_i+1}{N_i}-1)-(p+q)}(\tilde{E}_i^\circ,\Z)\quad\Rightarrow\quad H_c^{p+q}(\XX_m(f),\Z).
\ee
\end{theorem}

For hyperplane arrangements, again the situation is as easy as possible:

\begin{proposition}\label{propFlDeg} If $f:X=\C^n\rightarrow\C$ is a hyperplane multi-arrangement, the spectral sequence (\ref{eqSS}) degenerates at $E_1$ for any $m$-separating log resolution factoring through the canonical log resolution. 
\end{proposition}
\begin{proof} We start with some general facts about regular functions $f:X\ra \C$ on smooth complex varieties. 

Given a fixed log resolution, one obtains an $m$-separating log resolution by blowing up successively codimension-2 intersections $E_i\cap E_j$ with $i\ne j$ with $m_i+m_j\le m$, by the proof of \cite[Lemma 2.8]{Floer}. 

Suppose now that a fixed log resolution $\varphi:Y\ra X$ of $f$ satisfies that (\ref{eqELM}) is the decomposition into connected components and each component is irreducible. Let $\rho:Y'\ra Y$ be the blowup of $Y$ at $E_1\cap E_2$, assumed to be non-empty. Then $\varphi'=\varphi\circ \rho:Y'\ra X$ is also a log resolution of $f$, and we will show that (\ref{eqELM}) for $\varphi'$ is also the decomposition into connected components. 

Let $E_0'$ be the exceptional divisor of $\rho$. For $i\ge 1$, let $E_i'$ be the strict transform of $E_i$, and let $E'=\cup_{i=0}^sE_i'$. We denote by $N'_i$ the order of vanishing of $f$ along $E_i'$, and by $k_i'$ the order of vanishing of $K_{Y'/X}$ along $E_i'$. Thus $N_0'=N_1+N_2$,  $k_0'=k_1+k_2+1$, and $N_i'=N_i$, $k_i'=k_i$ for $i\ge 1$.

We will show that the two decompositions
$$
\bigsqcup_{\sum_{i=1}^s \nu_iN_i=m}\varphi_{\infty}(\Cont^{\nu}(E))= \bigsqcup_{\sum_{i=0}^s \nu'_iN'_i=m}\varphi'_{\infty}(\Cont^{\nu'}(E'))
$$
of the (non-restricted) contact locus $\Cont^m(f)$ are the same. 

To $\nu'=(\nu'_i)_{0\le i\le s}$ such that $\nu'\cdot N'=m$, associate $\nu=(\nu_0'+\nu_1',\nu_0'+\nu_2',\nu_3',\nu_4',\ldots)$ in $\N^s$. Then $\nu\cdot N=m$, as well. Moreover, $\rho_\infty(\Cont^{\nu'}(E'))\subset \Cont^\nu(E)$, and hence, $\varphi'_{\infty}(\Cont^{\nu'}(E'))\subset \varphi_{\infty}(\Cont^{\nu}(E))$. If $\Cont^{\nu'}(E')\neq\emptyset$, then a small calculation using Theorem \ref{eilamu} shows that both $\varphi'_{\infty}(\Cont^{\nu'}(E'))$ and $\varphi_{\infty}(\Cont^{\nu}(E))$ are irreducible cyclinders of same codimension $\nu_0'(k_1+k_2+2)+\nu_1'(k_1+1)+\nu_2'(k_2+1)+\ldots$, and hence $\varphi'_{\infty}(\Cont^{\nu'}(E'))$ is dense in  $\varphi_{\infty}(\Cont^{\nu}(E))$. But since $\Cont^m(f)$ is covered by the non-empty sets $\varphi'_{\infty}(\Cont^{\nu'}(E'))$, it follows that $\varphi'_{\infty}(\Cont^{\nu'}(E'))= \varphi_{\infty}(\Cont^{\nu}(E))$.

Using Proposition \ref{propCompELM}, we have therefore now shown that an $m$-separating log resolution of a hyperplane arrangement factoring through the canonical log resolution also satisfies that (\ref{eqELM}) is the decomposition into connected components and each component is irreducible. Restricting to the restricted $m$-contact locus $\Cont^{m,1}(f)$ of $f$ in $\mathcal L(X)$, this implies that in the decomposition from \cite[Lemma 2.1]{Floer} each term is both open and closed in the Zariski topology. The rest of the proof of \cite[Theorem 1.1]{Floer} is therefore drastically simplified and one obtains the degeneration at $E_1$ of the spectral sequence for $\XX_m(f)=\Cont^{m,1}(f)_m$ in $\mathcal L_m(X)$.
\end{proof}

\end{document}